\documentclass[10pt,a4paper,reqno]{amsart}
\usepackage{amsmath}
\usepackage{amssymb}
\usepackage{mathrsfs}
\newtheorem{theorem}{Theorem}[section]
\newtheorem{lemma}[theorem]{Lemma}

\newtheorem{remark}[theorem]{Remark}
\numberwithin{equation}{section}

\newcommand{\set}[1]{\{#1\}}
\newcommand{\abs}[1]{\left\vert#1\right\vert}
\newcommand{\EPT}[1]{\mathrm{E}[#1]}

\DeclareMathOperator*{\diff}{\mathrm{d}}
\usepackage{xcolor}
\usepackage{url}
\usepackage{float}
\usepackage{tikz}
\usepackage{ifpdf}
\ifpdf
\usepackage{graphicx}
\usepackage[pdftex,
hyperindex=true,
breaklinks=true,
pdfstartview=FitH,
colorlinks=true,
filecolor=blue,
anchorcolor=blue,
citecolor=blue,
urlcolor=blue,
citebordercolor=green,
filebordercolor=green,
linkbordercolor=green,
pdfcreator={LaTeX}]{hyperref}
\else
\usepackage[dvipdfm,
hyperindex=true,
breaklinks=true,
pdfstartview=FitH,
colorlinks=true,
filecolor=blue,
anchorcolor=blue,
citecolor=blue,
urlcolor=blue,
citebordercolor=green,
filebordercolor=green,
linkbordercolor=green,
pdfcreator={LaTeX}]{hyperref}
\fi
\begin{document}
\title[Eigenvalues for sub-fractional Brownian motion]{Asymptotics
of Karhunen-Lo{\`e}ve Eigenvalues for sub-fractional Brownian motion and its application}
\author{Chun-Hao CAI}
\address{School of Mathematics(Zhuhai), Sun Yat-sen University \newline
  \indent Tangjia Bay, Sun Yat-sen University\newline
  \indent Zhuhai, Guangdong, 519082, P. R. China.\newline
  \indent Email: \url{caichunhao2021@163.com}
}
\author{Jun-Qi HU}
\address{School of Mathematics, SHUFE\newline
\indent 777, Guoding Road,\newline
\indent Shanghai, 200433, P. R. China.\newline
\indent Email: \url{junqihu@shufe.edu.cn}
}
\author{Ying-Li WANG}
\address{School of Mathematics, SHUFE\newline
\indent 777, Guoding Road,\newline
\indent Shanghai, 200433, P. R. China.\newline
\indent Email: \url{naturesky1994@163.com}
}
\begin{abstract}
  In the present paper, the Karhunen-Lo{\`e}ve eigenvalues for a sub-fractional
  Brownian motion are considered. Rigorous large $n$
  asymptotics for those eigenvalues are shown, based on functional analysis
  method. By virtue of these asymptotics, along with some standard large
  deviations results, asymptotical estimates for the small $L^2$-ball
  probabilities for a sub-fractional Brownian motion are derived. By the way,
  asymptotic analysis on the Karhunen-Lo{\`e}ve eigenvalues for the corresponding
  ``derivative'' process is also established.
\end{abstract}
\maketitle
\section{Introduction}

The eigenproblem for a centred stochastic process $X=(X(t))_{t\in[0,1]}$ over a
probability space $(\Omega,\mathscr{F},P)$ with covariance function $K(s,t)=\EPT{X(s)X(t)}$
consists of finding all pairs $(\lambda,\varphi)$ satisfying the equation
\begin{equation}\label{egeqrtgsp}
  K\varphi=\lambda\varphi
\end{equation}
in $L^{2}([0,1])$, where the corresponding linear operator is defined by
\begin{equation}
  (K\varphi)(t)\triangleq\int_0^1K(s,t)\varphi(s)\diff s,\quad\forall t\in[0,1].
\end{equation}
If $K(s,t)$ is square integrable, then $K:L^2([0,1])\to L^2([0,1])$ is self-adjoint, positive
and compact. Hence, the eigenvalues $\set{\lambda_n}$ for the operator $K$ are
nonnegative, and converge to zero after arranged in decreasing order. The corresponding normalised
eigenfunctions $\set{\varphi_n}$ form a complete orthonormal basis in $L^2([0,1])$.

In addition, if $(X(t))_{t\in[0,1]}$ is a square integrable process with zero mean and
continuous covariance, there exists a Karhunen-Lo{\`e}ve expansion(cf.
\cite{tsp}). More precisely, it admits a representation over $[0,1]$ as a uniformly $L^2(\Omega)$-convergent series:
\begin{equation}
  X(t)=\sum_{n=0}^{\infty}\sqrt{\lambda_{n}}\xi_n\varphi_n(t),
\end{equation}
where $\set{\xi_n}$ are orthonormal(i.e. $\EPT{\xi_j\xi_k}=\delta_{jk}$) random variables in $L^2(\Omega)$ with zero mean.
Since the Karhunen-Lo{\`e}ve expansion is an influential tool in analysing
the properties of stochastic processes, $\set{\lambda_n}$ are also called Karhunen-Lo{\`e}ve
eigenvalues for $(X(t))_{t\in[0,1]}$.

There are many applications relevant to the eigenproblems for stochastic
processes: asymptotics of the small ball probabilities(cf. \cite{sb}), sampling from heavy tailed distributions(cf. \cite{VT2013PNERD}) and so on.

On most occasions, such kind of eigenvalues and eigenfunctions are notoriously
hard to find explicitly. One exception is for the standard Brownian motion $B=(B_t)_{t\in[0,1]}$, where
\begin{equation}
  \lambda_n=\frac1{(n+1/2)^2\pi^2}
\end{equation}
and
\begin{equation}\label{eigfunsforbm}
  \varphi_n(t)=\sqrt{2}\sin(n+1/2)\pi{t}
\end{equation}
for $n=0,1,2,\cdots$.
This problem can be easily solved by reducing \eqref{egeqrtgsp} to a simple
boundary value problem for an ordinary differential equation(cf. \cite{tsp}).

A widely used extension of Brownian motion is fractional Brownian motion $B^{H}=(B^{H}(t))_{t\in[0,1]}$. Its covariance function is
\begin{equation}
  K^H(s,t)=\frac12(s^{2H}+t^{2H}-\abs{s-t}^{2H}),
\end{equation}
where $H\in(0,1)$ is called its Hurst exponent. The case $H=\frac12$
corresponds to Brownian motion. There are some important properties of fractional Brownian motion. For examples, it
has self-similarity and stationary increments(cf. \cite{scfbm}). The eigenproblem
of fractional Brownian motion has been discussed in several papers(cf. \cite{KL,eigenfbm}).

The author in
\cite{KL} used functional analysis method, and obtained the asymptotics of the eigenvalues for
fractional Brownian motion. The following is just a rephrasing of one of his results:
\begin{theorem}[J. C. Bronski, 2003]\label{rstinbr2003}
  For the fractional Brownian motion with Hurst exponent $H\in(0,1)$,
  its Karhunen-Lo{\`e}ve eigenvalues satisfies the large $n$ asymptotics
  \begin{equation}
    \lambda_n
    =\frac{\sin(\pi{H})\Gamma(2H+1)}{(n\pi)^{2H+1}}+o(n^{-\frac{(2H+2)(4H+3)}{4H+5}+\delta})
  \end{equation}
  for every $\delta>0$, where $\Gamma$ denotes the usual Euler gamma function.
\end{theorem}

The authors in \cite{eigenfbm} converted the eigenproblem for fractional Brownian
motion into an integro-algebraic systems by using Laplace transform, and solved
it by taking the inversion of the Laplace transform. And some other processes
derived by Brownian motion like Brownian bridge(cf. \cite{gaussianbridge}), the
Ornstein-Uhlenbeck process (cf. \cite{fractionalOU}),
etc., can be solved in a similar way. Compared to \cite{KL}, the profile of
eigenpair analysed with the method in \cite{eigenfbm} is more complete and
accurate.

Similar to the fractional Brownian motion, the sub-fractional Brownian motion
also presents the properties of self-similarity and long-range dependence (when the
Hurst exponent $H>\frac12$). Different from the fractional Brownian motion, there
is an additional term $|t+s|^{2H}$ in its covariance and the increment is not
stationary as a result. From this point of view, it was expected that the idea in \cite{eigenfbm} could also work for
sub-fractional Brownian motion. However, it
seems that it DOES NOT work for sub-fractional Brownian motion because of loss of some translation structure.

Consequently, the main results in this paper are based on the idea in \cite{KL}. It should
be pointed out that there are some flaws in \cite{KL}. To some extent, the
results(see Remark \ref{flawinKLl} in Section 4.) in this paper are supplements and corrections of the ones in \cite{KL}.

This paper is organised as follows. In the next section, the asymptotics of eigenvalues
for sub-fractional Brownian motion and its derivative process are going to be stated. As an
application of those results, the small ball estimates for sub-fractional Brownian
motion will be presented in Section 3., but its proof will be omitted since it's just a
duplication of the one in \cite{KL}. Section 4 will be concluded with
the details of the proofs of the main results.

\section{The main results}

Sub-fractional Brownian motion(sfBm) $B_{sub}^H=(B_{sub}^H(t))_{t\in[0,1]}$ is a centred long-range dependence Gaussian process. Like fractional Brownian motion, its covariance function is
\begin{equation}
  K_{sub}^H(s,t)=s^{2H}+t^{2H}-\frac12[(s+t)^{2H}+\abs{s-t}^{2H}]
\end{equation}
with an exponent $H\in(0,1)$. The case $H=\frac12$ also corresponds to Brownian motion. To some extent, sfBm is intermediate between Brownian motion and fractional Brownian motion(cf. \cite{sfbm}). This is reflected in the nonstationarity and correlation of the increments and the covariance of the non-overlapping intervals. The increments on non-overlapping intervals are more weakly correlated than fractional Brownian motion and the covariance decays polynomially at a higher rate.

In this paper, the eigenproblems for the following two operators
\begin{align}
  (K_{sub}^H\varphi)(t)&=\int_0^1(s^{2H}+t^{2H}-\frac12[(s+t)^{2H}+\abs{s-t}^{2H}])\varphi(s)\diff s,\label{cvopforsfbm}\\
  (\widetilde{K}_{sub}^H\varphi)(t)&=\int_0^1H(2H-1)(\abs{s-t}^{2H-2}-(s+t)^{2H-2})\varphi(s)\diff s\label{cvopfordrpsfbm}
\end{align}
are studied. The operator in \eqref{cvopforsfbm} for $H\in(0,1)$ is related to sfBm itself, and the one in \eqref{cvopfordrpsfbm} for $H\in(\frac{1}{2},1)$ corresponds to
the formal derivative of the sfBm. In fact, the operator $\widetilde{K}_{sub}^H$ determines the correlation
structure of Wiener integrals of square integrable deterministic functions through the formula
\begin{equation}
  \EPT{\int_0^1f\diff B_{sub}^H\int_0^1g\diff B_{sub}^H}
  =\int_0^1f(t)\,(\widetilde{K}_{sub}^Hg)(t)\diff t.
\end{equation}

To the best of the authors' knowledge, those eigenproblems haven't been
rigorously considered before. Borrowed the idea from \cite{KL}, rough
asymptotics of eigenvalues of sub-fractional Brownian motion are derived:
\begin{theorem}\label{1strstinpp}
  The Karhunen-Lo{\`e}ve eigenvalues of sub-fractional Brownian motion with
  exponent $H\in(0,1)$ satisfies
  \begin{itemize}
  \item Case $0<H<\frac{-1+\sqrt{74}}{8}$:
    \begin{equation}
      \lambda_n=\frac{\gamma_H}{n^{2H+1}}+o(n^{-\frac{(2H+2)(4H+3)}{4H+5}+\delta})
    \end{equation}
    for every $\delta>0$ and $n\gg1$;
  \item Case $\frac{-1+\sqrt{74}}{8}\le{H}<1$:
    \begin{equation}
      \lambda_n=\frac{\gamma_H}{n^{2H+1}}+O(n^{-3})
    \end{equation}
  \end{itemize}
  for every $n\gg1$, where $\gamma_H=\frac{2\sin(\pi{H})\Gamma(2H+1)}{\pi^{2H+1}}$.
\end{theorem}

Specifically, given an orthonormal basis in $L^{2}([0,1])$, the operator $K_{sub}^H$ in
\eqref{cvopforsfbm} over $L^{2}([0,1])$ is of a representation as a linear
operator over $\ell^{2}$, which is essentially an infinite-dimensional matrix.
Asymptotic analysis on matrix elements are performed, based on some technical
lemmas, some of which(see Lemma \ref{rfmlmrtkl}) are improvements of the ones in \cite{KL}. Afterwards,
Theorem \ref{1strstinpp} is obtained in terms of the theory of compact operators.
However, asymptotics in Theorem \ref{1strstinpp} are rough by simple observation
or through numerical simulation, although the details of the simulation are not provided here.

Next conclusion is about the eigenvalues of the derivative process of
sub-fractional Brownian. Unlike \cite{eigenfbm}, the case $H\in(0,\frac12)$ is skipped.
\begin{theorem}
  The Karhunen-Lo{\`e}ve eigenvalues of the derivative process of sub-fractional Brownian with $H\in(\frac12,1)$ satisfies
  \begin{equation}
    \lambda_n=\frac{\kappa_H}{n^{2H-1}}+o(n^{-\frac{2H(4H-1)}{4H+1}+\delta})
  \end{equation}
  for every $\delta>0$ and $n\gg1$,
  where $\kappa_H=\frac{2\sin(\pi{H})\Gamma(2H+1)}{\pi^{2H-1}}$.
\end{theorem}

\section{An application: Small $L^2$-ball estimate}
Small ball estimate is an interesting topic in probability theory, and also has
important applications in statistical mechanic models. It yields estimates of the
probability that some stochastic process $X=(X(t))_{t\in[0,1]}$ will lie inside a ball of radius
$\varepsilon$ in a certain given norm $||\cdot||.$ As for the $L^2([0,1])$-norm, if
$X$ is a centred Gaussian process with continuous covariance, there holds
\begin{equation}
  ||X||_{L^2}^2=\int_0^1X(t)^2\diff t=\sum_{n=0}^\infty\lambda_n\xi_n^2,
\end{equation}
where $\set{\xi_n}$ are i.i.d. $N(0,1)$ random variables. Just as pointed out in \cite{KL}, a crucial quantity to
derive small ball estimate for fractional Brownian motion is the determinant
\begin{equation}
  D_H(\lambda)=\prod_{n=0}^\infty(1+2\lambda\lambda_n)
\end{equation}
which is a variant of Fredholm determinant of $K_{H}$.

Now, the small $L^2([0,1])$-ball
estimate of sfBm is going to be carried out. Let ${D_{sub}^H(\lambda)}$ be the
corresponding determinant with respect to sfBm. Note that the dominant terms of
the eigenvalues of sfBm(see Theorem \ref{1strstinpp}) are just double of the ones of fractional
Brownian motion(see Theorem \ref{rstinbr2003}) when $n\gg1$. Through a slight modification of the proof of
Corollary 1 in Appendix C of
\cite{KL}, the logarithmic asymptotics of ${D_{sub}^H(\lambda)}$ read
\begin{lemma}
  There holds
  \begin{itemize}
  \item Case $H\in(0,\frac{-1+\sqrt{74}}{8})$:
    \begin{equation}
\log(D_{sub}^H(\lambda))=\frac{(4\sin(\pi{H})\Gamma(2H+1))^{\frac1{2H+1}}}{\sin(\frac{\pi}{2H+1})}\lambda^{\frac1{2H+1}}+o(\lambda^{\frac{4H+4}{(4H+5)(2H+1)}+\delta})
\end{equation}
for every $\delta>0$ and $\lambda\gg1$;
  \item Case $H\in[\frac{-1+\sqrt{74}}{8},1)$:
    \begin{equation}     \log(D_{sub}^H(\lambda))=\frac{(4\sin(\pi{H})\Gamma(2H+1))^{\frac1{2H+1}}}{\sin(\frac{\pi}{2H+1})}\lambda^{\frac1{2H+1}}+O(\lambda^{\frac{2H-1}{2H+1}})
    \end{equation}
    for every $\lambda\gg1$.
\end{itemize}
\end{lemma}
Thereafter, the small $L^2([0,1])$-ball estimate of sfBm can be directly established by using standard large
deviations calculation and de Bruijn's exponential Tauberian theorem(cf. \cite{BGTT1987RV}), which
is exactly the same procedure as the one in \cite{KL}.
\begin{theorem}
  For $0<\varepsilon\ll1$, the small ball probability
  $P(||B_{sub}^H||_{L^{2}}^2\le\varepsilon)$ of a sub-fractional Brownian
  motion satisfies
  \begin{itemize}
  \item Case $H\in(0,\frac{-1+\sqrt{74}}{8})$:
  \begin{equation}
    \log(P(||B_{sub}^H||_{L^{2}}^2\le\varepsilon))
    =-H\left(\frac{2\sin(\pi{H})\Gamma(2H+1)}{((2H+1)\sin(\frac{\pi}{2H+1}))^{2H+1}}\right)^{\frac1{2H}}\varepsilon^{-\frac1{2H}}+o(\varepsilon^{-\frac{4H+4}{(4H+5)2H}+\delta})
  \end{equation}
  for every $\delta>0$;
\item Case $H\in[\frac{-1+\sqrt{74}}{8},1)$:
  \begin{equation}
    \log(P(||B_{sub}^H||_{L^{2}}^2\le\varepsilon))
    =-H\left(\frac{2\sin(\pi{H})\Gamma(2H+1)}{((2H+1)\sin(\frac{\pi}{2H+1}))^{2H+1}}\right)^{\frac1{2H}}\varepsilon^{-\frac1{2H}}+o(\varepsilon^{-\frac{2H-1}{2H}})
  \end{equation}
\end{itemize}
\end{theorem}

\section{Proofs of the Main Results}

Throughout this section, the eigenfunctions $\set{\varphi_{n}}$ in
\eqref{eigfunsforbm} are chosen as an orthonormal basis in $L^2([0,1])$.
Therefore, any bounded linear operator $K$ over $L^2([0,1])$ is one-to-one
corresponding to the operator $A$ over $\ell^{2}$ with the same operator norm. The
linear operator $A$ over $\ell^{2}$ is essentially an infinite-dimensional
matrix $(A_{m,n})$, whose element is given by
\begin{equation}
  A_{m,n}=\int_0^1\int_0^1K(x,y)\varphi_{m}(x)\varphi_{n}(y)\diff x\diff y.
\end{equation}
Actually, such kind of mapping $K$ to $A$ is a topologically isomorphism. It implies that, if
$K$ is compact(Hilbert-Schmidt etc.) in $L^2([0,1])$, then $A=(A_{m,n})$ is also
compact(Hilbert-Schmidt etc.) in $\ell^{2}$; Vice versa.

For the sake of simplicity, denote
\begin{equation}
  m^*=(m+\frac12)\pi,\quad n^*=(n+\frac12)\pi,\quad m,n=0,1,2,\cdots
\end{equation}
in the sequel. Now, the eigenfunctions for Brownian motion could be rewritten
as
\begin{equation}
  \varphi_{n}(t)=\sqrt{2}\sin{(n^{*}t)},\quad n=0,1,2,\cdots.
\end{equation}

It's ready to prove the main results in this paper.

\subsection{Proof of Theorem \ref{1strstinpp}}

Here, the eigenproblem is $K_{sub}^H\varphi=\lambda\varphi$. The proof will be
finished in 5 steps. In the first 4 steps, $H\in(\frac12,1)$ is imposed
temporarily, but this condition will be dropped off in Remark \ref{drpffimpcnd}.

{\it Step 1.} Obviously, for the operator $K_{sub}^H$ in $L^2([0,1])$, there exists a linear operator $A_{sub}^H=((A_{sub}^H)_{m,n})$
in $\ell^{2}$ with the same operator norm, whose element is
\begin{equation}\label{merofttopforsfbm}
  (A_{sub}^H)_{n,m}
  =\int_0^1\int_0^12\left[x^{2H}+y^{2H}-\frac12\left((x+y)^{2H}+\abs{x-y}^{2H}\right)\right]\sin(n^*x)\sin(m^*y)\diff x\diff y.
\end{equation}
It's easy to see that for every $m,n=0,1,2,\cdots$, there holds
\begin{equation}\label{merof1stiopkforsfbm}
  (A_{sub}^H)_{n,m}
  =\frac{2H(2H-1)}{n^*m^*}\int_0^1\int_0^1(\abs{x-y}^{2H-2}-(x+y)^{2H-2})\cos(n^*x)\cos(m^*y)\diff x\diff y.
\end{equation}
by utilising the integration by parts since $H>\frac{1}{2}$. Splitting the right hand side of
\eqref{merof1stiopkforsfbm} into two integrals, the linear operator $A_{sub}^H$ has a
decomposition  $A_{sub}^H=2A-A^{(1)}$, where their corresponding elements share the
same relations, i.e., $(A_{sub}^H)_{n,m}=2A_{n,m}-A_{n,m}^{(1)}$ for every $m,n=0,1,2,\cdots$.

{\it Step 2.}
It's worthy mentioning that the linear operator $A=(A_{m,n})$ in $\ell^{2}$ with its elements of the forms
\begin{equation}\label{merof1stiopkforfbm}
  A_{n,m}=\frac{H(2H-1)}{n^*m^*}\int_0^1\int_0^1\abs{x-y}^{2H-2}\cos(n^*x)\cos(m^*y)\diff x\diff y
\end{equation}
has been discussed in \cite{KL}. There exists a
decomposition $A=D+O$, where the linear operator $D$ and $O$ are corresponding to
infinite-dimensional matrices whose elements are respectively as a leading order
diagonal piece and a higher order off-diagonal piece
of $(A_{m,n})$. Accurately speaking, according to the proof of Theorem 1 in
Appendix A in \cite{KL}, there hold
\begin{align}  D_{n,m}&=\left(\frac{\sin(\pi{H})\Gamma(2H+1)}{n^{*2H+1}}+O(\frac{1}{n^{2(H+1)}})\right)\delta_{n,m},\label{dpofiopforfbm}\\
O_{n,m}&=\frac{\cos(\pi{H})\Gamma(2H+1)}{n^*m^*(n^*+(-1)^{n+m+1}m^*)}\left(\frac{1}{n^{*2H-1}}+(-1)^{n+m+1}\frac{1}{m^{*2H-1}}\right)+O(\frac{1}{n^{2}m^{2}})\label{odpofiopforfbm}
\end{align}
for $m,n\gg1$, where $O_{n,n}=0$ for every $n=0,1,2,\cdots$.
\begin{remark}\label{flawinKLl}
  By applying Lemma \ref{rfmlmrtkl} below, it is accidently found that the remainder order of $D_{n,n}$ in
  \cite{KL}(or see \eqref{dpofiopforfbm} above) is not correct, while
  \eqref{rcdpofiopforfbm} is the right one in stead.
\end{remark}
In order to obtain the exact rate of convergence of $O$, \eqref{odpofiopforfbm}
could be rewritten as
\begin{equation}
  O_{n,m}
  =\left\{
    \begin{array}{l}
\frac{\cos(\pi{H})\Gamma(2H+1)}{n^*m^*(n^*-m^*)}(\frac{1}{n^{*2H-1}}-\frac{1}{m^{*2H-1}})+O(\frac{1}{n^{2}m^{2}})\quad m+n\quad {\rm even}\\
  \frac{\cos(\pi{H})\Gamma(2H+1)}{n^*m^*(n^*+m^*)}(\frac{1}{n^{*2H-1}}+\frac{1}{m^{*2H-1}})+O(\frac{1}{n^{2}m^{2}})\quad m+n\quad {\rm
    odd}.\\
    \end{array}
    \right.
\end{equation}
It's sufficient to discuss the case of $m>n\gg1$ because of the symmetry with
respect to the subscripts $m$ and $n$ in \eqref{merof1stiopkforfbm}. Whenever $m+n$ is
even or not, it is clear that
\begin{equation}
  \frac{1}{n^*m^*(n^*\pm m^*)}(\frac{1}{n^{*2H-1}}\pm\frac{1}{m^{*2H-1}})
  =\pm\frac{1}{m^{*2}n^{*2H}}\frac{1\pm(\frac{n^{*}}{m^{*}})^{2H-1}}{1\pm\frac{n^{*}}{m^{*}}}
\end{equation}
which leads to\footnote{The notation $f\asymp g$ means $f$ and $g$ are the same order of magnitude.}
\begin{equation}\label{offdpofiopforfbm}
  O_{n,m}\asymp\frac{1}{m^{2}n^{2H}},\quad m>n\gg1
\end{equation}
by noticing the boundedness of $f(t)=\frac{1\pm t^{2H-1}}{1\pm t}$ in $t\in(0,1)$.

{\it Step 3.}
It's time to deal with the linear operator $A^{(1)}=(A_{n,m}^{(1)})$, whose
element is
\begin{equation}
  A_{n,m}^{(1)}=\frac{2H(2H-1)}{n^*m^*}\int_0^1\int_0^1(x+y)^{2H-2}\cos(n^*x)\cos(m^*y)\diff x\diff y,
\end{equation}
Simply decompose $A^{(1)}$ into $A^{(1)}=D^{(1)}+O^{(1)}$ as done in Step 2, where
$D^{(1)}_{n,m}\triangleq A^{(1)}_{n,m}\delta_{n,m}$, and $O^{(1)}_{n,m}\triangleq A^{(1)}_{n,m}-D^{(1)}_{n,m}$.

{\it Step 3.1.}
To calculate the elements of $A^{(1)}=(A_{n,m}^{(1)})$, firstly
divide the square $[0,1]\times[0,1]$ into two sub-domains $I_{1}$, $I_{2}$(see
Figure \ref{domainforA1}),
where $I_1$ represents the triangle enclosed by the lines $x=0$, $y=0$ and
$x+y=1$; $I_2$ the triangle enclosed by $x=1$, $y=1$ and $x+y=1$. It leads to
\begin{equation}
  \begin{aligned}
    &\int_0^1\int_0^1(x+y)^{2H-2}\cos(n^*x)\cos(m^*y)\diff x\diff y\\
    =&\left(\iint_{I_{1}}+\iint_{I_{2}}\right)(x+y)^{2H-2}\cos(n^*x)\cos(m^*y)\diff x\diff y.
  \end{aligned}
\end{equation}
Through the change of variables
\begin{equation}
  \begin{cases}
    u=x+y\\
    v=x-y,
  \end{cases}
\end{equation}
it maps $I_1$ and $I_2$ to $J_1$ and $J_2$ respectively. By changing of variables
in double integration, it implies that
\begin{equation}\label{imdidttforclot}
  \begin{aligned}
    &2\int_0^1\int_0^1(x+y)^{2H-2}\cos(n^*x)\cos(m^*y)\diff x\diff y\\
    =&\left(\int_0^1\diff u\int_{-u}^u+\int_1^2\diff u\int_{u-2}^{-u+2}\right)u^{2H-2}\cos(n^*(\frac{u+v}{2}))\cos(m^*(\frac{u-v}{2}))\diff v.
  \end{aligned}
\end{equation}
It's convenient to denote two integral terms on the right hand side of \eqref{imdidttforclot} by
$Q^{1}_{n,m}$ and $Q^{2}_{n,m}$.
\begin{figure}[htbp]
  \centering
\begin{tikzpicture}
  \draw[->] (-1,0) -- (2,0);
  \draw (2.3,0) node {$x$};
  \draw[->] (0,-1) -- (0,2);
  \draw (0,2.3) node {$y$};
  \draw (0,1)--(1,0);
  \draw (0,1)--(1,1);
  \draw (1,1)--(1,0);
  \draw (0.25,0.25) node {$I_1$};
  \draw (0.75,0.75) node {$I_2$};
  \draw [->,very thick] (2.55,0) -- (2.85,0);
  \draw[->] (3,0) -- (4,0);
  \draw[densely dotted] (4,0) -- (6,0);
  \draw[->] (6,0) -- (7,0);
  \draw (7.3,0) node {$u$};
  \draw[->] (4,-1) -- (4,2);
  \draw (4,2.3) node {$v$};
  \draw (4,0)--(5,1);
  \draw (4,0)--(5,-1);
  \draw (5,-1)--(5,1);
  \draw (5,1)--(6,0);
  \draw (5,-1)--(6,0);
  \draw (4.5,0) node {$J_1$};
  \draw (5.5,0) node {$J_2$};
\end{tikzpicture}
\caption{Domains for calculating $A^{(1)}_{n,m}$.}
\label{domainforA1}
\end{figure}
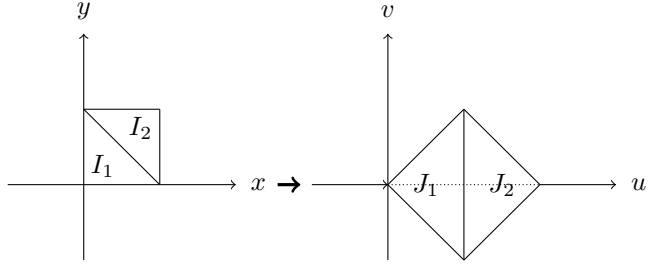
Combined with the formulae for trigeometric functions, it implies that
\begin{align}
  Q^{1}_{n,m}&=\frac{1}{2}\int_0^1u^{2H-2}\diff u\int_{-u}^u\left(\cos(\frac{m^*+n^*}{2}u-\frac{m^*-n^*}{2}v)+\cos(\frac{m^*-n^*}{2}u-\frac{m^*+n^*}{2}v)\right)\diff v,\label{1staiforliop}\\
  Q^{2}_{n,m}&=\frac{1}{2}\int_1^2u^{2H-2}\diff u\int_{u-2}^{-u+2}\left(\cos(\frac{m^*+n^*}{2}u-\frac{m^*-n^*}{2}v)+\cos(\frac{m^*-n^*}{2}u-\frac{m^*+n^*}{2}v)\right)\diff v.
\end{align}
Substituting the above identities into $A^{(1)}$, it can be deduced that
\begin{equation}
  A^{(1)}_{n,m}=\frac{H(2H-1)}{n^*m^*}\left(Q^{1}_{n,m}+Q^{2}_{n,m}\right).
\end{equation}

{\it Step 3.2.}
Since the singularity among the integrands in $A^{(1)}_{m,n}$ only
occurs at $(0,0)$, it seems that the contribution of $Q^{1}_{m,n}$ should be much
greater than the one of $Q^{2}_{m,n}$. To see it, the following integral identities will be needed.

\begin{lemma}\label{biiforasyma}
  Let $a$ be a real number, there hold
  \begin{align}
    \int_1^2u^{a}\cos(\omega u)\diff u&=\frac{2^{a}\sin(2\omega)-\sin{\omega}}{\omega}+O(\frac{1}{\omega^{2}}),\\
    \int_1^2u^{a}\sin(\omega u)\diff u&=-\frac{2^{a}\cos(2\omega)-\cos{\omega}}{\omega}+O(\frac{1}{\omega^{2}})
  \end{align}
  for $\omega\gg1$.
\end{lemma}
\begin{proof}
  It's necessary to prove the first identity since the second could be proved in
  a similar way. Noticing that
  \begin{equation}
    \frac{1}{\omega}\int_1^2\diff(u^{a}\sin(\omega u))=\frac{2^{a}\sin(2\omega)-\sin{\omega}}{\omega}
  \end{equation}
  it implies that
  \begin{equation}
    \int_1^2u^{a}\cos(\omega u)\diff u
    =\frac{2^{a}\sin(2\omega)-\sin{\omega}}{\omega}-\frac{a}{\omega}\int_1^2u^{a-1}\sin(\omega u)\diff u.
  \end{equation}
  Combining with the second mean value theorem for Riemann integrals, the desired
  result is obtained.
\end{proof}

To calculate $Q^{2}_{n,m}$, the order of $\int_1^2u^{2H-2}\sin(m^*u)\diff u$
needs to be estimated. By setting $a=2H-2$ and $\omega=m^*$,
Lemma \ref{biiforasyma} gives
\begin{align}
  \int_1^2u^{2H-2}\cos(m^*u)\diff u&=\frac{(-1)^{m+1}}{m^*}+O(\frac1{m^{2}}),\label{12cos2}\\
  \int_1^2u^{2H-2}\sin(m^*u)\diff u&=\frac{2^{2H-2}}{m^*}+O(\frac1{m^{2}})\label{12sin2nonremainder}
\end{align}
for $m\gg1$. Base on same idea, the order of remainder term can be improved. For
example, it's true that
\begin{align}\label{12cos1}
  \begin{aligned}
    \int_1^2u^{2H-1}\cos(m^*u)\diff u
    =&-\frac{2H-1}{m^*}\int_1^2u^{2H-2}\sin(m^*u)\diff u-\frac{(-1)^m}{m^*}\\
    =&\frac{(-1)^{m+1}}{m^*}+O(\frac1{m^{2}})
  \end{aligned}
\end{align}
for $m\gg1$. Moreover, there holds
\begin{equation}
  \frac{1}{m^*-n^*}\int_1^2u^{2H-2}(\sin(m^*u)-\sin(n^*u))\diff u
  =\frac{1}{m^*-n^*}\left(\frac{2^{2H-2}}{m^*}-\frac{2^{2H-2}}{n^*}\right)
  +O(\frac1{mn})\label{12sin2remainder}
\end{equation}
for $m>n\gg1$.

\begin{lemma}\label{rfmlmrtkl}
  If $a\in(0,1)$, there hold
  \begin{align}
    \int_0^1x^{a-1}\cos(\omega x)\diff x&=\frac{\Gamma(a)\cos(\frac\pi2a)}{\omega^a}+\frac{\sin{\omega}}{\omega}+O(\frac{1}{\omega^{2}}),\\
    \int_0^1x^{a-1}\sin(\omega x)\diff x&=\frac{\Gamma(a)\sin(\frac\pi2a)}{\omega^a}-\frac{\cos{\omega}}{\omega}+O(\frac{1}{\omega^{2}})
  \end{align}
  for $\omega\gg1$.
\end{lemma}

The proof of Lemma \ref{rfmlmrtkl} is postponed in the next subsection.
By setting $a=2H-1$ and $\omega=m^*$, the two identities in Lemma \ref{rfmlmrtkl} are turned into
\begin{align}
  \int_0^1x^{2H-2}\cos(m^*x)\diff x&=\frac{\Gamma(2H-1)\sin(\pi{H})}{m^{*2H-1}}
                                     +\frac{(-1)^m}{m^*}+O(\frac{1}{m^{2}}),\label{01cos2}\\
  \int_0^1x^{2H-2}\sin(m^*x)\diff x&=-\frac{\Gamma(2H-1)\cos(\pi{H})}{m^{*2H-1}}+O(\frac{1}{m^{2}}).\label{01sin2nonremainder}
\end{align}
for $m\gg1$.

On the one hand, it was mentioned in Remark \ref{flawinKLl} that the asymptotics in
\eqref{dpofiopforfbm} are not correct. As a matter of fact, using Lemma \ref{rfmlmrtkl},
correct ones could be deduced. That is, the diagonal part of the matrix corresponding to fractional Brownian motion can be revised as
\begin{equation}\label{rcdpofiopforfbm}
  D_{n,n}=\frac{\sin(\pi{H})\Gamma(2H+1)}{n^{*2H+1}}+\frac{(-1)^n}{n^{*3}}+O(\frac{1}{n^{4}})
\end{equation}
for $n\gg1$.

On the other hand, by using the integration by parts(see The proof of Lemma
\ref{rfmlmrtkl}) and the second mean value
theorem for Riemann integrals, it's valid that
\begin{equation}
  \int_0^1u^{2H-1}\cos(m^*u)\diff u=\frac{(-1)^m}{m^*}+\frac{\Gamma(2H)\cos(\pi{H})}{m^{*2H}}+O(\frac1{m^{3}}),\label{01cos1}
\end{equation}
for $m\gg1$. Furthermore, there holds
\begin{equation}
  \frac{1}{m^*-n^*}\int_0^1x^{2H-2}\sin(m^*x)\diff x
  =-\frac{\Gamma(2H-1)\cos(\pi{H})}{m^*-n^*}(\frac1{m^{*2H-1}}-\frac1{n^{*2H-1}})+O(\frac{1}{mn})\label{01sin2remainder}
\end{equation}
for $m>n\gg1$.

{\it Step 3.3.}
Calculate $Q^{1}_{n,m}$ and $Q^{2}_{n,m}$ in the case of $m>n\gg1$. Firstly, Using the
fundamental theorem for Riemann integrals in \eqref{1staiforliop}, it implies that
\begin{equation}
  \begin{aligned}
    Q^{1}_{n,m}=&\frac1{m^*+n^*}\int_0^1u^{2H-2}(\sin(m^*u)+\sin(n^*u))\diff u\\
    &+\frac1{m^*-n^*}\int_0^1u^{2H-2}(\sin(m^*u)-\sin(n^*u))\diff u
  \end{aligned}
\end{equation}
which gives
\begin{equation}
  \begin{aligned}    Q^{1}_{n,m}=&-\frac{\Gamma(2H-1)\cos(\pi{H})}{m^*+n^*}(\frac1{m^{*2H-1}}+\frac1{n^{*2H-1}}+O(\frac{1}{m^{2}})+O(\frac{1}{n^{2}}))\\    &-\frac{\Gamma(2H-1)\cos(\pi{H})}{m^*-n^*}(\frac1{m^{*2H-1}}-\frac1{n^{*2H-1}})+O(\frac{1}{mn})
  \end{aligned}
\end{equation}
in terms of Lemma \ref{rfmlmrtkl}(see \eqref{01sin2nonremainder} and
\eqref{01sin2remainder}). Observing that for $m>{n}\gg1$,
\begin{equation}
  \frac{1}{m^*+n^*}(O(\frac{1}{m^{2}}+O(\frac{1}{n^{2}}))=O(\frac1{mn})
\end{equation}
it means that
\begin{equation}
  Q^{1}_{n,m}=-\frac{\Gamma(2H-1)\cos(\pi{H})}{m^*+n^*}(\frac1{m^{*2H-1}}+\frac1{n^{*2H-1}})
  -\frac{\Gamma(2H-1)\cos(\pi{H})}{m^*-n^*}(\frac1{m^{*2H-1}}-\frac1{n^{*2H-1}})+O(\frac1{mn})
\end{equation}
i.e. the order of $Q^{1}_{n,m}$ is the same as $m^{-1}n^{-(2H-1)}$.

Next goal is to calculate $Q^{2}_{n,m}$. It's clear that
\begin{equation}
  \begin{aligned}
    Q^{2}_{n,m}=&\frac{(-1)^{m+n+1}}{m^*+n^*}\int_1^2u^{2H-2}(\sin(m^*u)+\sin(m^*u))\diff u\\
    &+\frac{(-1)^{m+n+1}}{m^*-n^*}\int_1^2u^{2H-2}(\sin(m^*u)-\sin(n^*u))\diff u\\
  \end{aligned}
\end{equation}
which leads to
\begin{equation}
  \begin{aligned}    Q^{2}_{n,m}=&\frac{(-1)^{m+n+1}}{m^*+n^*}(\frac{2^{2H-2}}{m^*}+\frac{2^{2H-2}}{n^*}+O(\frac1{m^{2}})+O(\frac1{n^{2}}))\\
    &+\frac{(-1)^{m+n+1}}{m^*-n^*}(\frac{2^{2H-2}}{m^*}-\frac{2^{2H-2}}{n^*})+O(\frac1{mn})
  \end{aligned}
\end{equation}
in terms of Lemma \ref{biiforasyma}(see \eqref{12sin2nonremainder} and \eqref{12sin2remainder}).
After all, it gives that
\begin{equation}
  Q^{2}_{n,m}=O(\frac1{mn})
\end{equation}
which verifies that the contribution of $Q^{2}_{n,m}$ is smaller than the one of
$Q^{1}_{n,m}$.

Since
\begin{equation}\label{offdpofiopforext}
  O_{n,m}^{(1)}=A_{n,m}^{(1)}
  =\frac{H(2H-1)}{n^*m^*}\left(Q^{1}_{n,m}+Q^{2}_{n,m}\right)
\end{equation}
for $m\neq n$, $O_{n,m}^{(1)}$ is the same order as $m^{-2}n^{-2H}$ for $m>{n}\gg1$.

{\it Step 3.4.}
Calculate $Q^{1}_{n,m}$ and $Q^{2}_{n,m}$ in the case of $m=n\gg1$. At first,
\eqref{1staiforliop} could be transformed into
\begin{equation}
  Q^{1}_{m,m}=\int_0^1u^{2H-1}\cos(m^*u)\diff u+\frac1{m^*}\int_0^1u^{2H-2}\sin(m^*u)\diff u
\end{equation}
which gives
\begin{equation}
  Q^{1}_{m,m}=\frac{(-1)^m}{m^*}+\frac{\cos(\pi{H})}{m^{*2H}}(\Gamma(2H)-\Gamma(2H-1))+O(\frac{1}{m^{3}})
\end{equation}
by virtue of \eqref{01cos1} and \eqref{01sin2nonremainder}. Secondly, it's valid that
\begin{equation}
  Q^{2}_{m,m}=-\int_1^2u^{2H-1}\cos(m^*u)\diff u+2\int_1^2u^{2H-2}\cos(m^*u)\diff u+\frac1{m^*}\int_1^2u^{2H-2}\sin(m^*u)\diff u
\end{equation}
which implies that
\begin{equation}
  Q^{2}_{m,m}=\frac{(-1)^{m+1}}{m^*}+O(\frac1{m^{2}})
\end{equation}
by using \eqref{12cos2}, \eqref{12sin2nonremainder} and \eqref{12cos1}.
Since
\begin{equation}
  D_{m,m}^{(1)}=A_{m,m}^{(1)}
  =\frac{H(2H-1)}{m^{*2}}\left(Q^{1}_{m,m}+Q^{2}_{m,m}\right)
\end{equation}
it implies that
\begin{equation}\label{dpofiopforext}
  D^{(1)}_{m,m}=\frac{(H-1)\cos(\pi{H})\Gamma(2H+1)}{m^{*2H+2}}+O(\frac1{m^{4}})
\end{equation}
i.e. $D_{m,m}^{(1)}$ is the same order as $m^{-2H-2}$ for $m\gg1$.

{\it Step 4.}
Summarise all asymptotic information for $A_{sub}$. Noting that
$A_{sub}=2A-A^{(1)}$ and $A^{(1)}=D^{(1)}+O^{(1)}$, $A_{sub}$ has also a decomposition
$A_{sub}=D_{sub}+O_{sub}$ just like the linear operator $A$ in Step 2, if $D_{sub}=2D-D^{(1)}$ and
$O_{sub}=2O-O^{(1)}$ are set.

The orders of the elements of $A_{sub}$ are as follows: As for the diagonal piece,
combined \eqref{rcdpofiopforfbm} with \eqref{dpofiopforext}, it gives
\begin{equation}\label{dpofttopforsfbm}
  \begin{aligned}
    (D_{sub})_{m,m}=&\frac{2\sin(\pi{H})\Gamma(2H+1)}{m^{*2H+1}}+\frac{(-1)^m}{m^{*3}}\\
    &+\frac{(H-1)\cos(\pi{H})\Gamma(2H+1)}{m^{*2H+2}}+O(\frac{1}{m^{4}})
  \end{aligned}
\end{equation}
for $m\gg1$; As for the off-diagonal piece, noticing \eqref{offdpofiopforfbm} and
\eqref{offdpofiopforext}, it implies
\begin{equation}
  (O_{sub})_{n,m}\asymp \frac{1}{m^{2}n^{2H}}
\end{equation}
for $m>n\gg1$.

\begin{remark}\label{drpffimpcnd}
  During processing the proof of Theorem \ref{1strstinpp}, $H\in(\frac12,1)$ is
  imposed. In fact, $A_{n,m}$(see \eqref{merofttopforsfbm}) is holomorphic with respect to the variable $H$ in
  $(0,1)$, so are $D_{n,m}$ and $O_{n,m}$. Moreover, the first three terms on the right hand side of
  \eqref{dpofttopforsfbm} are holomorphic in $H\in(0,1)$, so is the remaining term in \eqref{dpofttopforsfbm}.
  In terms of the principle of analytic continuation, \eqref{dpofttopforsfbm} is still
  valid for $H\in(0,1)$. Same argument works for off-diagonal piece in the case
  of $H\in(0,1)$.
\end{remark}

{\it Step 5.}
It's clear that $D_{sub}$ is self-adjoint, positive and compact in $\ell^{2}$.
For any fixed $\beta\in(0,1)$, $D_{sub}^{\beta}$ is well-defined by spectral
decomposition theorem. Hence, $O_{sub}$ can be turned into
\begin{equation}
  O_{sub}=D_{sub}^{\beta}\widehat{O}_{sub}D_{sub}^{\beta},
\end{equation}
where
\begin{equation}
  \widehat{O}_{sub}=D_{sub}^{-\beta}O_{sub}D_{sub}^{-\beta}.
\end{equation}
The order of the elements of $D^{-\beta}_{sub}$ is $m^{\beta(2H+1)}$ for $m\gg1$, so the
order of the ones of $\widehat{O}_{sub}$ is $n^{(2H+1)\beta-2}m^{(2H+1)\beta-2H}$
for $m>n\gg1$. If $\beta\in(0,\frac12)$, the elements of $\widehat{O}_{sub}$ are square summable. Therefore, $\widehat{O}_{sub}$ is a Hilbert-Schmidt operator(and thus compact). The eigenvalues of $\widehat{O}_{sub}$ are square summable and thus\;(arranged in order of decreasing magnitude) satisfy
\begin{equation}
  \abs{\lambda_n(\widehat{O}_{sub})}\lesssim n^{-\frac12}.
\end{equation}

Finally, recall the following two results(cf. \cite{PS}).
\begin{lemma}[Porter and Stirling]\label{egvlforcgtfofcop}
  If $T$, $K$ are compact and $K$ is self-adjoint then the eigenvalues of $T^*KT$ satisfy
  \begin{equation}
    \abs{\lambda_n(T^*KT)}\le\min_{j\in\set{1,\cdots,n}}\abs{\lambda_j(K)}\;\lambda_{n-j+1}(T^*T).
  \end{equation}
\end{lemma}

\begin{lemma}[Porter and Stirling]\label{egvlforamtfofcop}
  If $K_1,K_2$ are compact and self-adjoint then we have
  \begin{equation}
    \lambda_n(K_1+K_2)\le\min_{j\in\set{1,\cdots,n}}\abs{\lambda_{n-j+1}(K_1)}+\abs{\lambda_j(K_2)}.
  \end{equation}
\end{lemma}

Given any $\delta\in(0,\frac12)$, by setting $\beta=\frac12-\delta$, it's true that
\begin{align*}
  \abs{\lambda_n(O_{sub})}
  =&\abs{\lambda_n(D_{sub}^\beta\widehat{O}_{sub}D_{sub}^\beta)}
  \leq\abs{\lambda_{n-j}(\widehat{O}_{sub})}\abs{\lambda_j(D_{sub}^{2\beta})}\\
  \lesssim&n^{-\frac12}n^{-2\beta(2H+1)}=n^{-2H-\frac32+(4H+2)\delta}
\end{align*}
in terms of Lemma \ref{egvlforcgtfofcop}. Since $\delta\in(0,\frac12)$ is
arbitarily chosen, the above inequality can be rewritten as
\begin{equation}
  \abs{\lambda_n(O_{sub})}\lesssim{n^{-2H-\frac32+\delta}}.
\end{equation}

Now, Lemma \ref{egvlforamtfofcop} yields
\begin{equation}
  \begin{aligned}
    \lambda_n(A_{sub})
    &\le\abs{\lambda_{n-n^\alpha}(D_{sub})}+\abs{\lambda_{n^\alpha}(O_{sub})}\\
    &\le\frac{\gamma_H}{n^{2H+1}}(1+\frac{n^\alpha}{n-n^\alpha})^{2H+1}+O((n-n^\alpha)^{-3})+O(n^{-\alpha(2H+\frac32-\delta)})\\
    &=\frac{\gamma_H}{n^{2H+1}}(1+(2H+1)\frac{n^\alpha}{n-n^\alpha}+O(\frac{n^{2\alpha}}{(n-n^\alpha)^2}))\\
    &+O((n-n^\alpha)^{-3})+O(n^{-\alpha(2H+\frac32-\delta)})\\
    &=\frac{\gamma_H}{n^{2H+1}}+O(n^{-2H-2+\alpha})+O(n^{-3})+O(n^{-\alpha(2H+\frac32-\delta)})
    \end{aligned}
\end{equation}
by setting $K_1=D_{sub}$, $K_2=O_{sub}$ and $j=n^\alpha$.
Letting $2H+2-\alpha=\alpha(2H+\frac32)$(i.e. $\alpha=\frac{2H+2}{2H+\frac52}$)
and making use of the arbitrariness of $\delta\in(0,\frac{1}{2})$, it implies that
\begin{equation}
  \lambda_n(A_{sub})\le\frac{\gamma_H}{n^{2H+1}}+o(n^{-\frac{(2H+2)(4H+3)}{4H+5}+\delta})+O(n^{-3}).
\end{equation}
There are two cases:
\begin{enumerate}
\item If $\frac{(2H+2)(4H+3)}{4H+5}<3$(i.e. $0<H<\frac{-1+\sqrt{74}}{8}$),
  there holds
  \begin{equation}
    \lambda_n(A_{sub})\le\frac{\gamma_H}{n^{2H+1}}+o(n^{-\frac{(2H+2)(4H+3)}{4H+5}+\delta});
  \end{equation}
\item If $\frac{(2H+2)(4H+3)}{4H+5}\ge3$(i.e. $\frac{-1+\sqrt{74}}{8}\le{H}<1$),
  there holds
  \begin{equation}
    \lambda_n(A_{sub})\le\frac{\gamma_H}{n^{2H+1}}+O(n^{-3}).
  \end{equation}
\end{enumerate}
Repeating the above argument with $K_1=A_{sub}$, $K_2=-O_{sub}$ gives
\begin{enumerate}
\item If $0<H<\frac{-1+\sqrt{74}}{8}$, there holds
  \begin{equation}
    \lambda_n(A_{sub})\ge\frac{\gamma_H}{n^{2H+1}}+o(n^{-\frac{(2H+2)(4H+3)}{4H+5}+\delta});
  \end{equation}
\item If $\frac{-1+\sqrt{74}}{8}\le{H}<1$, there holds
  \begin{equation}
    \lambda_n(A_{sub})\ge\frac{\gamma_H}{n^{2H+1}}+O(n^{-3}).
  \end{equation}
\end{enumerate}

The proof is completed.\qed\\

\subsection{Proof of Lemma \ref{rfmlmrtkl}}

It's sufficient to prove the first identity. The proof of the second identity is similar to the first
one. First, by changing of variable in integration, it can be deduced that
\begin{equation}
  \int_0^1x^{a-1}\cos(\omega x)\diff x=\frac{1}{\omega^a}\left(\int_0^{+\infty}t^{a-1}\cos{t}\diff t-\int_{\omega}^{+\infty}t^{a-1}\cos{t}\diff t\right).
\end{equation}
On the one hand, by using contour integration, it's easy to verify
\begin{equation}
  \int_0^{\infty}t^{a-1}\cos{t}\diff t=\Gamma(a)\cos(\frac\pi2a).
\end{equation}
On the other hand, by using the integration by parts and the
second mean value theorem for Riemann integrals, it's valid that
\begin{equation}
  \begin{aligned}
    \int_\omega^{\infty}t^{a-1}\cos{t}\diff t
    =&-\omega^{a-1}\sin{\omega}-(a-1)\int_\omega^{\infty}t^{a-2}\sin{t}\diff t\\
    =&-\omega^{a-1}\sin{\omega}+O(\omega^{a-2}).
  \end{aligned}
\end{equation}
The proof is completed.\qed

\subsection{Proof of Theorem 2.2}

Following the lines in the proof of Theorem \ref{1strstinpp}, it's easy to justify Theorem 2.2.
Here the sketch of its proof will be given and the different part from the steps
in the proof of Theorem \ref{1strstinpp} will be emphasised. Step 3.2. in the proof
of Theorem \ref{1strstinpp} is skipped since technical lemmas have been exhibited
there.

Formally speaking, the covariance function $\widetilde{K}_{sub}^H$ is the ``mixed partial
derivative'' of $K_{sub}^H$. From the point of view of the general
white noise theory(cf. \cite{scfbm}), the sfBm is the integral process of the one
related to $\widetilde{K}_{sub}^H$ in rigorous sense, since
\begin{equation}
  \int_0^s\int_0^t\widetilde{K}_{sub}^H(x,y)\diff x\diff y=K_{sub}^H(s,t)
\end{equation}
if $H>\frac12$. Hence, it's reasonable to study the eigenproblem $\widetilde{K}_{sub}^H\varphi=\lambda\varphi$.

{\it Step 1.}
The matrix element related to the linear operator $\widetilde{K}_{sub}^H$ becomes
\begin{equation}
  (\widetilde{A}_{sub})_{n,m}=2H(2H-1)\int_0^1\int_0^1(\abs{x-y}^{2H-2}-(x+y)^{2H-2})\sin(n^*x)\sin(m^*y)\diff x\diff y.
\end{equation}
By spiliting the right hand side of the above identity into two integrals,
$(\widetilde{A}_{sub})_{n,m}$ has a decomposition
$(\widetilde{A}_{sub})_{n,m}=2\widetilde{A}_{n,m}-\widetilde{A}^{(1)}_{n,m}$,
where $\widetilde{A}_{n,m}$ corresponds to the part of fractional Brownian
noise(cf. \cite{KL,eigenfbm}).

To calculate $\widetilde{A}_{n,m}$, through imitating the method in \cite{KL},
$[0,1]\times[0,1]$ can be represented by a parallelogram $V_1$ minus two
triangles $V_2$ and $V_3$(see Figure \ref{domainfortldA}), where $V_1$ is enclosed by the lines $y=0$, $y=1$,
$y-x=1$ and $y-x=-1$; $V_2$ enclosed by $x=0$, $y=0$ and $y-x=1$; $V_3$
enclosed by $x=1$, $y=1$ and $y-x=-1$.
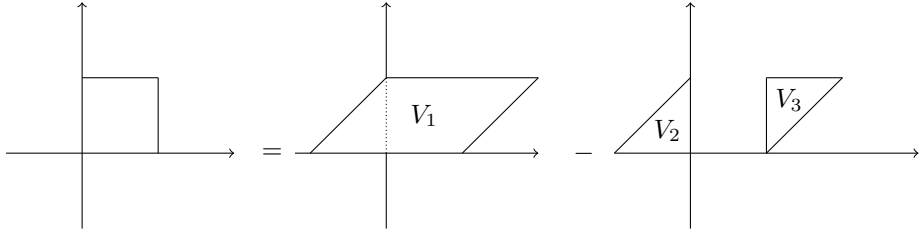
\begin{figure}[htbp]
  \centering
  \begin{tikzpicture}
    \draw[->] (-1,0) -- (2,0);
    \draw[->] (0,-1) -- (0,2);
    \draw (0,1)--(1,1);
    \draw (1,1)--(1,0);
    \draw (2.5,0) node {$=$};
    \draw[->] (2.8,0) -- (6,0);
    \draw[->] (4,1) -- (4,2);
    \draw (4,-1)--(4,0);
    \draw[densely dotted] (4,0)--(4,1);
    \draw (3,0)--(4,1);
    \draw (4,1)--(6,1);
    \draw (5,0)--(6,1);
    \draw (6.6,0) node {$-$};
    \draw[->] (7,0) -- (11,0);
    \draw[->] (8,-1) -- (8,2);
    \draw (7,0)--(8,1);
    \draw (9,0)--(9,1);
    \draw (9,1)--(10,1);
    \draw (9,0)--(10,1);
    \draw (4.5,0.5) node {$V_1$};
    \draw (7.7,0.3) node {$V_2$};
    \draw (9.3,0.7) node {$V_3$};
  \end{tikzpicture}
  \caption{Domains for calculating $\widetilde{A}_{n,m}$.}
  \label{domainfortldA}
\end{figure}
By denoting
\begin{align}
  \widetilde{R}_{n,m}^1&=\iint_{V_1}|x-y|^{2H-2}\sin(n^*x)\sin(m^*y)\diff x\diff y\\
  \widetilde{R}_{n,m}^{2,3}&=\iint_{V_2\cup{V_3}}|x-y|^{2H-2}\sin(n^*x)\sin(m^*y)\diff x\diff y,
\end{align}
it's clear that
\begin{equation}
  \widetilde{A}_{n,m}=H(2H-1)(\widetilde{R}_{n,m}^1-\widetilde{R}_{n,m}^{2,3}).
\end{equation}
By changing the variables $u=y-x$, $v=y$ in double integral, it can be deduced that
\begin{equation}
  \widetilde{R}_{nm}^1=\int_0^1\sin(m^*v)\sin(n^*v)\diff v\int_{-1}^1|u|^{2H-2}\cos(n^*u)\diff u.
\end{equation}
To calculate $\widetilde{R}_{n,m}^{2,3}$, first of all, perform the mapping $V_3$ to
$V_2$ by changing of variables $x'=1-x$, $y'=1-y$. Then the integral over this
region is greatly simplified under the change of variables $u=x+y$, $v=x-y$, which gives
\[
\widetilde{R}_{n,m}^{2,3}=\frac{(-1)^{m+n+1}}2\int_0^1u^{2H-2}\diff u\int_{-u}^u\cos(n^*\frac{u-v}{2}+(-1)^{m+n}m^*\frac{v+u}{2})\diff v.
\]

To calculate $\widetilde{A}^{(1)}_{n,m}$, two sub-domains $I_{1}$, $I_{2}$ are
chosen just as done in Step 3.1. of Section 4.1. Designating
\begin{align}
  \widetilde{Q}_{n,m}^1=\iint_{I_1}(x+y)^{2H-2}\sin(n^*x)\sin(m^*y)\diff x\diff y\\
  \widetilde{Q}_{n,m}^2=\iint_{I_2}(x+y)^{2H-2}\sin(n^*x)\sin(m^*y)\diff x\diff y
\end{align}
it gives
\begin{equation}
  \widetilde{A}^{(1)}_{n,m}=2H(2H-1)(\widetilde{Q}_{n,m}^1+\widetilde{Q}_{n,m}^{2}).
\end{equation}

{\it Step 2.}
Now, it's time to extract diagonal and off-diagonal information from
$(\widetilde{A}_{sub})_{n,m}$. By setting
$(\widetilde{D}_{sub})_{n,m}=(\widetilde{A}_{sub})_{n,m}\delta_{n,m}$, and
$(\widetilde{O}_{sub})_{n,m}=(\widetilde{A}_{sub})_{n,m}-\widetilde{D}_{n,m}$, a
decomposition $(\widetilde{A}_{sub})_{n,m}=(\widetilde{D}_{sub})_{n,m}+(\widetilde{O}_{sub})_{n,m}$
is obtained. Moreover, there hold
\begin{align}
  (\widetilde{D}_{sub})_{n,n}=&2H(2H-1)(\widetilde{R}_{n,n}^1-\widetilde{R}_{n,n}^{2,3}-\widetilde{Q}_{n,n}^1-\widetilde{Q}_{n,n}^2),\\
  (\widetilde{O}_{sub})_{n,m}=&2H(2H-1)(\widetilde{R}_{n,m}^1-\widetilde{R}_{n,m}^{2,3}-\widetilde{Q}_{n,m}^1-\widetilde{Q}_{n,m}^2),\quad n\neq m.
\end{align}
The details for handling $\widetilde{A}_{n,m}$ part will be emphasised, but the
ones for $\widetilde{A}_{n,m}^{(1)}$ will be omitted except for the conclusions.

{\it Step 2.1.} Calculate $\widetilde{A}_{n,m}$ and $\widetilde{A}_{n,m}^{(1)}$
in the case of $m>n\gg1$. Simple calculations show that
\begin{align}
  \widetilde{R}_{n,m}^1&=0.\\
  \widetilde{R}_{n,m}^{2,3}&=\frac{(-1)^{m+n+1}}{m^*+(-1)^{m+n}n^*}\int_0^1u^{2H-2}(\sin(m^*u)+(-1)^{m+n}\sin(n^*u))\diff u.
\end{align}
Processing as Step 3. in Section 4.1, it's no trouble to verify
\begin{align}
  &\begin{aligned}
    \widetilde{Q}_{n,m}^1=&-\frac1{2(m^*+n^*)}\int_0^1u^{2H-2}(\sin(m^*u)+\sin(n^*u))\diff u\\
    &+\frac1{2(m^*-n^*)}\int_0^1u^{2H-2}(\sin(m^*u)-\sin(n^*u))\diff u.
  \end{aligned}\\
  &\begin{aligned}
    \widetilde{Q}_{n,m}^2
    =&\frac{(-1)^{m+n}}{2(m^*+n^*)}\int_1^2u^{2H-2}(\sin(m^*u)+\sin(n^*u))\diff u\\
    &+\frac{(-1)^{m-n}}{2(m^*-n^*)}\int_1^2u^{2H-2}(\sin(m^*u)-\sin(n^*u))\diff u.
  \end{aligned}
\end{align}
Using \eqref{12sin2nonremainder},\eqref{12sin2remainder},
\eqref{01sin2nonremainder} and \eqref{01sin2remainder}, it implies that
\begin{equation}
\begin{aligned}
&\widetilde{R}_{n,m}^1-\widetilde{R}_{n,m}^{2,3}-\widetilde{Q}_{n,m}^1-\widetilde{Q}_{n,m}^2\\
=&\frac{3(-1)^{m+n}\Gamma(2H-1)\cos(\pi{H})}{2(m^*+(-1)^{m+n}n^*)}(\frac1{m^{*2H-1}}+(-1)^{m+n}\frac1{n^{*2H-1}})\\
&+(-1)^{m+n}\frac{\Gamma(2H-1)\cos(\pi{H})}{2(m^*-(-1)^{m+n}n^*)}(\frac{1}{m^{*2H-1}}-(-1)^{m+n}\frac1{n^{*2H-1}})+O(\frac1{mn}).
\end{aligned}
\end{equation}
Using the same techniques as {\it Step 3.3.} in Lemma 4.3, it's obvious that $\widetilde{R}_{n,m}^1-\widetilde{R}_{n,m}^{2,3}-\widetilde{Q}_{n,m}^1-\widetilde{Q}_{n,m}^2{\asymp}m^{-1}n^{2H-1}.$

{\it Step 2.2.} Calculate $\widetilde{A}_{n,m}$ and $\widetilde{A}_{n,m}^{(1)}$
in the case of $m=n\gg1$. It's easy to check that
\begin{align}
  &\begin{aligned}
    \widetilde{R}_{m,m}^1
    =&\int_0^1\sin(m^*v)\sin(m^*v)\diff v\int_{-1}^1|u|^{2H-2}\cos(m^*u)\diff u\\
    =&2(\frac{\Gamma(2H-1)\sin(\pi{H})}{m^{*2H-1}}+\frac{(-1)^m}{m^*}+O(\frac{1}{m^{2}}))
  \end{aligned}\\
  &\begin{aligned}
    \widetilde{R}_{m,m}^{2,3}
    =&-\frac1{m^*}\int_0^1u^{2H-2}\sin(m^*u)\diff u\\
    =&\frac{\Gamma(2H-1)\cos(\pi{H})}{m^{*2H}}+O(\frac{1}{m^{3}}).
  \end{aligned}
\end{align}
Following the similar procedures as {\it Step 3.} in Section 4.1, there hold
\begin{align}
\widetilde{Q}_{m,m}^1=&\frac12\int_0^1u^{2H-1}\cos(m^*u)\diff u-\frac1{2m^*}\int_0^1u^{2H-2}\sin(m^*u)\diff u.\\
\widetilde{Q}_{m,m}^2=&\frac1{2m^*}\int_1^2u^{2H-2}\sin(m^*u)\diff u+\frac12\int_1^2u^{2H-1}\cos(m^*u)\diff u-\int_1^2u^{2H-2}\cos(m^*u)\diff u.
\end{align}
Using \eqref{12cos2}, \eqref{12sin2nonremainder},
\eqref{12cos1} and \eqref{01cos1}, it leads to
\[
\begin{aligned}
&\widetilde{R}_{n,n}^1-\widetilde{R}_{n,n}^{2,3}-\widetilde{Q}_{n,n}^1-\widetilde{Q}_{n,n}^2\\
=&\frac{2\Gamma(2H-1)\sin(\pi{H})}{n^{*2H-1}}-\frac{(2\Gamma(2H)+3\Gamma(2H-1))\cos(\pi{H})}{2n^{*2H}}+O(\frac{1}{n^{2}}).
\end{aligned}
\]

{\it Step 2.3.} Summarise the asymptotic information for the matrix elements of
$\widetilde{A}_{sub}$. The asymptotics for diagonal piece of $(\widetilde{A}_{sub})_{n,m}$
is
\begin{equation}
  (\widetilde{D}_{sub})_{m,m}=\frac{2\Gamma(2H+1)\sin(\pi{H})}{m^{*2H-1}}-\frac{(4H+1)\Gamma(2H+1)\cos(\pi{H})}{2m^{*2H}}+O(\frac{1}{m^{2}})
\end{equation}
if $m\gg1$, and the ones for off-diagonal piece is
\begin{equation}
  (\widetilde{O}_{sub})_{n,m}\asymp\frac{1}{mn^{2H-1}}
\end{equation}
if $m>n\gg1$.

{\it Step 3.}
Noticing that $\widetilde{D}_{sub}$ is self-adjoint and positive, given any
$\beta\in(0,1)$, $\widetilde{O}_{sub}$ can also be written as
$\widetilde{O}_{sub}=\widetilde{D}_{sub}^{\beta}\widehat{O}_{der}\widetilde{D}_{sub}^{\beta}$
with $\widehat{O}_{der}=\widetilde{D}_{sub}^{-\beta}
\widetilde{O}_{sub}\widetilde{D}_{sub}^{-\beta}$. Since
the order of the elements of $\widetilde{D}_{sub}^{-\beta}$ is $m^{\beta(2H-1)}$
when $m\gg1$, the order of the ones of $\widehat{O}_{der}$ is
$m^{(2H-1)\beta-1}n^{(2H-1)\beta-2H+1}$ when $m>n\gg1$. If $\beta\in(0,\frac12)$,
it's easy to verify whether the elements of $\widehat{O}_{der}$ are square
summable. In fact,
\[
\begin{aligned}
\sum_{m>n}(\widehat{O}_{der})_{n,m}^2
\lesssim&\sum_{m>n}m^{2(2H-1)\beta-2}n^{2(2H-1)\beta-4H+2}\\
=&\sum_{n}n^{2(2H-1)\beta-4H+2}\sum_{m=n+1}^{\infty}m^{2(2H-1)\beta-2}\\
\lesssim&\sum_{n}n^{4(2H-1)\beta-4H+1}.
\end{aligned}
\]
The square summability of $\widehat{O}_{der}$ is verified since
$4(2H-1)\beta-4H+1\in(-4H+1,-1)$ when $\beta\in(0,\frac12)$. Therefore,
$\widehat{O}_{der}$ is a Hilbert-Schmidt operator(and thus compact). Using Lemma
\ref{egvlforcgtfofcop}, it is immediately obtained that
\begin{equation}
  |\lambda_n(\widetilde{O}_{sub})|\lesssim{n^{-2H+\frac12+\delta}}.
\end{equation}

Setting $K_1=\widetilde{D}_{sub}$, $K_2=\widetilde{O}_{sub}$ and $j=n^\alpha$ in Lemma
\ref{egvlforamtfofcop}, it can be deduced that
\begin{equation}
  \lambda_n(\widetilde{A}_{sub})\le\frac{\kappa_H}{n^{2H-1}}+O(n^{-2H+\alpha})+O(n^{-\alpha(2H-\frac12-\delta)}).
\end{equation}
Choosing $2H-\alpha=\alpha(2H-\frac12)$(i.e. $\alpha=\frac{2H}{2H+\frac12}$), it
implies that
\begin{equation}
  \lambda_n(\widetilde{A}_{sub})\le\frac{\kappa_H}{n^{2H-1}}+o(n^{-\frac{2H(4H-1)}{4H+1}+\delta}).
\end{equation}
Repeating the argument with $K_1=\widetilde{A}_{sub}$, $K_2=-\widetilde{O}_{sub}$ gives
\begin{equation}
  \lambda_n(\widetilde{A}_{sub})\ge\frac{\kappa_H}{n^{2H-1}}+o(n^{-\frac{2H(4H-1)}{4H+1}+\delta}).
\end{equation}

The proof is completed.\qed

\section{Acknowledgement}
Ying-Li WANG, one of the authors, would like to thank Professor Ping HE and Assistant Professor
Qing-Hua WANG for their patient discussion.

\end{document}